\documentclass{amsart}
\usepackage[utf8]{inputenc}
\usepackage[english]{babel}
\usepackage{amsfonts}
\usepackage[11pt]{moresize}
\usepackage{pgfplots}
\usepackage{xfrac}

\usepackage{palatino}
\usepackage{amscd}
\usepackage[breaklinks,bookmarksopen,bookmarksnumbered]{hyperref}
\usepackage{amssymb}
\usepackage{mathtools}
\usepackage{tikz}
\usepackage{tikz-cd}
\usepackage{mathtools}
\usepackage{amsmath}
\usepackage{amsthm}
\usepackage{hyperref}
\usepackage{float}
\usetikzlibrary{calc, positioning,decorations.markings,arrows}

\usepackage{graphics}
\usepackage{color}
\usepackage{tabularx,colortbl}
\usetikzlibrary{positioning}
\usetikzlibrary{decorations.markings,arrows}
\usetikzlibrary{calc}
\usepackage{amsthm}

\theoremstyle{plain}
\newtheorem{trm}{Theorem}[section]
\newtheorem{lm}[trm]{Lemma}
\newtheorem{prop}[trm]{Proposition}

\theoremstyle{definition}
\newtheorem{defi}[trm]{Definition}
\newtheorem{rmk}[trm]{Remark}

\begin{document}
	\title{Maximal gonality on strata of differentials and uniruledness of strata in low genus}
	\author{Andrei Bud}
	\address{Humboldt-Universit\"at zu Berlin, Institut f\"ur Mathematik, Rudower Chausee 25
		\hfill \newline\texttt{}
		\indent 12489 Berlin, Germany} 
	\email{{\tt andreibud95@protonmail.com}}
    \begin{abstract}
    	 We prove that for a generic element in a nonhyperelliptic component of an abelian stratum $\mathcal{H}_g(\mu)$ in genus $g$, the underlying curve has maximal gonality. We extend this result to the case of quadratic strata when the partition $\mu$ has positive entries. As a consequence we deduce that all nonhyperelliptic components of $\mathcal{H}_9(\mu)$ are uniruled when $\mu$ is a positive partition of 16 and all nonhyperelliptic components of $\mathcal{H}^2_g(\mu)$ are uniruled when $\mu$ is a positive partition of $4g-4$ and either $3\leq g\leq5$ or $g=6$ and $l(\mu)\geq 4$.
    \end{abstract}
\maketitle
\section{Introduction} 
For positive integers $g$ and $k$ and a partition $\mu = (m_1,\ldots, m_n)$ of $(2g-2)k$, the stratum
$\mathcal{H}_g^k(\mu)$ in $\mathcal{M}_{g,n}$ parametrizing $k$-canonical divisors of type $\mu$ is defined as 
\[\mathcal{H}_g^k(\mu) = \left\{ [C,x_1,\ldots, x_n] \in \mathcal{M}_{g,n} \ | \ \mathcal{O}_C(\sum_{i=1}^n m_ix_i) \cong \omega_C^k \right\} \]
In the case $k=1$ we drop the superscript in the notation. 

The problem of understanding the components of $\mathcal{H}^k_g(\mu)$ has received a lot of attention in recent years, see \cite{abelcompo}, \cite{Boissy}, \cite{Lanneau} and \cite{ChenM}. One important question is to understand the relative position of $\mathcal{H}^k_g(\mu)$ with respect to the gonality stratification of $\mathcal{M}_g$.   
Our result provides an answer to this question for the cases $k=1$ and $k=2$ when the partition $\mu$ has only positive entries. 

\begin{trm} \label{trmabel}
	Let $g \geq 3$, $n\geq 1$ and $\mu$ a length $n$ positive partition of $2g-2$. If $[C,x_1,\ldots, x_n]$ is a generic point of a nonhyperelliptic component of $\mathcal{H}_g(\mu)$, then the curve $C$ has maximal gonality.  
\end{trm} 
We have a similar result for the quadratic case. 
\begin{trm} \label{trmquad}
	Let $g \geq 3$, $n\geq 1$ and $\mu$ a length $n$ positive partition of $4g-4$. If $[C,x_1,\ldots, x_n]$ is a generic point of a nonhyperelliptic component of $\mathcal{H}^2_g(\mu)$, then the curve $C$ has maximal gonality.  
\end{trm}

It is to be expected that Theorems \ref{trmabel} and \ref{trmquad} may be instrumental in computing the Kodaira dimension of strata of $k$-canonical divisors. In this respect the divisor in $\mathcal{M}_g$ parametrizing curves $[C]$ having less than maximal gonality was used in \cite{KodMg} to compute the Kodaira dimension of $\mathcal{M}_g$ for odd $g \geq 25$. The pull-back of this divisor was used in \cite{FarLud} to show that the moduli $\mathcal{R}_g$ of Prym varieties of dimension $g-1$ has maximal Kodaira dimension when $g\geq 17$.

Theorem \ref{trmabel} is used to compute the Kodaira dimension of strata in low genus. We recall the method used in \cite{BAR18} and based on an idea that first appeared in \cite{FarVer}.
If $[C,x_1,\ldots,x_n]$ is a generic point of $\mathcal{H}_9(\mu)$ and $C$ is embedded as a hyperplane section of a K3 surface $S$ of genus $g$ in $\mathbb{P}^g$ then there exists a $g-2$ dimensional linear subspace $\Lambda_{\mu}$ of $\mathbb{P}^g$ satisfying 
\[ \Lambda_{\mu}\cdot S = \sum_{i=1}^n m_ix_i \]
By intersecting the pencil of hyperplanes containing $\Lambda_{\mu}$ with the K3 surface $S$ we get a rational curve 
\[\mathbb{P}^1 \rightarrow \overline{\mathcal{H}}_9(\mu)\]
passing through $[C,x_1,\ldots, x_n]$, hence implying uniruledness. 

In genus 9 the restriction $l(\mu)\geq 7$ on the length of the partition in \cite{BAR18} appears as it was not known whether a generic point of $[C,x_1,\ldots, x_n]$ in the stratum $\mathcal{H}_9(\mu)$ admits such an embedding in a K3 surface when $l(\mu)\leq 6$. Mukai proved that each curve of genus 9 of maximal gonality can be realized as a hyperplane section of a K3 surface, see \cite[Theorem A]{MukK3}. As a consequence we get the following result:  

\begin{trm} \label{unirul}
	Let $\mu$ a positive partition of 16. Then all nonhyperelliptic components of $\mathcal{H}_9(\mu)$ are uniruled. 
\end{trm}

A similar approach works for quadratic strata $\mathcal{H}^2_g(\mu)$ when $3\leq g\leq 6$. Assume that for a generic point $[C,x_1,\ldots,x_n]$ of a component of $\mathcal{H}^2_g(\mu)$ the curve $C$ is a hyperplane section in a del Pezzo surface $S$ embedded in $\mathbb{P}^{3g-3}$ by the linear system $|-2K_S|$. Because the restriction of $|-2K_S|$ to $C$ is $|2K_C|$, the same method applies to this case. Due to Theorem \ref{trmquad} we can drop some of the extra conditions in \cite{BAR18} and show: 

\begin{trm} \label{delPezzo}
	Let $3\leq g\leq 5$ and $\mu$ a positive partition of $4g-4$. Then all nonhyperelliptic components of the quadratic stratum $\mathcal{H}^2_g(\mu)$ are uniruled. If $g=6$ and $\mu$ is a positive partition of 20 of length $l(\mu)\geq 4$, then all nonhyperelliptic components of $\mathcal{H}^2_g(\mu)$ are uniruled.
\end{trm}

In what follows, we will reduce Theorems \ref{trmabel} and \ref{trmquad} to the case when the partition $\mu$ has length one and then use a degeneration argument to prove this case. 
For this, we will need the compactification of $\mathcal{H}_g^k(\mu)$ inside $\overline{\mathcal{M}}_{g,n}$. Results in this direction appear in \cite{Gen18}, \cite{FP18} and \cite{Sch18} and a complete description was provided in \cite[Corollary 1.4]{DaweiAbComp} and \cite[Theorem 1.5]{Daweik-diffcomp}. Another important ingredient is the description of the irreducible components of $\mathcal{H}_g^k(\mu)$. For $k=1, g\geq 4$, if $\mu$ is positive with only even entries, the space $\mathcal{H}_g(\mu)$ has two nonhyperelliptic components depending on the parity of the spin structure and has one if $\mu$ has at least one odd entry, as seen in \cite[Theorem 1]{abelcompo}. For $k=2$ and $\mu$ positive, except for some sporadic strata in genus 4, the space $\mathcal{H}^2_g(\mu)$ has a unique irreducible component denoted $\mathcal{Q}_g(\mu)$ that is nonhyperelliptic and not an abelian component. This description appears in \cite[Theorem 1.1 and Theorem 1.2]{Lanneau}  and is completed in \cite[Theorem 1.2]{ChenM}.

The description of the parity of the spin structure appearing in \cite[Example 6.1]{corn} is also needed, but as the task of determining the parity will be trivial, this will be tacitly used. As a matter of convention, we say that a partition $\mu$ is positive when all its entries are positive.  
 
\textbf{Acknowledgements} I would like to thank my advisor Gavril Farkas for suggesting this beautiful topic. I am grateful to Johannes Schmitt for highlighting to me the importance of admissible covers and their applications. I am also grateful to the anonymous referee for carefully reading the paper and the many suggested improvements. I was financially supported by the Berlin Mathematical School, within the framework of the Cluster of Excellence.

	\section{Properties of admissible covers}  \label{sec1}

	In what follows we will use the compactification of the Hurwitz scheme with admissible covers, appearing in \cite{KodMg} and generalized in \cite{Diaz} to prove some general properties for them. Afterwards, we particularize these results to strata of $k$-canonical divisors. We start with the following definition 
	\begin{defi} \normalfont Let $f\colon Y \rightarrow \Gamma$ an admissible cover with $Y$ a curve of compact type and $\Gamma$ a curve of genus $0$. Let $y$ be a node of $Y$ joining two irreducible components $Y_1$ and $Y_2$. Denote by $\mathbb{P}_1$ and $\mathbb{P}_2$ the images of $Y_1$ and $Y_2$ through $f$. If $\mathbb{P}$ is a component of $\Gamma$, we say that $\mathbb{P}$ is in the same direction as $y$ with respect to $Y_1$ if there exist a chain $\mathbb{P}_1,\mathbb{P}_2,\ldots, \mathbb{P}_n = \mathbb{P}$ of irreducible components of $\Gamma$ such that $\mathbb{P}_i$ and $\mathbb{P}_{i+1}$ share a node for every $1\leq i\leq n-1$ and $\mathbb{P}_i \neq \mathbb{P}_j$ for any pair of indices $i\neq j$. 
	\end{defi}
	
	In order to understand the numerical conditions respected by the admissible cover, we provide a similar definition for the components of $Y$. 
	\begin{defi} \normalfont In the notations above, we define the subcurve $Y_{y;Y_1}$ of $Y$ to be the union of irreducible components $Z$ for which we have a chain of irreducible components $Y_1,Y_2,\ldots, Y_n = Z$ without repetitions such that every component share a node with the previous one in the chain. In other words, $Y_{y;Y_1}$ is the maximal connected subcurve of $Y$ satisfying $Y_1 \cap Y_{y;Y_1} = \left\{ y\right\}$.	
	\end{defi} 

    In the above notations we have the following 
    \begin{lm}\label{admcyc} Let $\mathbb{P}$ be a component of $\Gamma$ and denote by $c$ and $c_1$ the degree of $f_{|Y_{y;Y_1}}$ over $\mathbb{P}$ and $\mathbb{P}_1$ respectively. Then, if $\mathbb{P}$ is in the same direction as $y$ with respect to $Y_1$ we have the equality $c = c_1 + \mathrm{ord}_y(f_{|Y_1})$. Otherwise we have $c= c_1$. 
    \end{lm}
\begin{proof} This is just a double count of the number of preimages of the nodes of $\Gamma$ with respective multiplicities given by the ramification order.  
\end{proof}

In what follows, by a $g_k^1$ we mean the map to $\mathbb{P}^1$ it defines. Using Lemma \ref{admcyc} we are able to prove the following proposition, relating the possible limit linear series of a generic element in two different irreducible components.

\begin{prop} \label{eventoodd} Let $k\geq 1, i\geq 1$ and consider two irreducible subspaces $Z_1 \subseteq \mathcal{M}_{2i,1}$ and $Z_2 \subseteq \mathcal{M}_{2i+1,1}$ such that a generic element $[C,x]$ of $Z_1$ has maximal gonality and does not admit a $g^1_{i+j}$ with ramification order at $x$ greater or equal to $2j$ for $1\leq j\leq k$. Moreover, assume there exists a 2-pointed elliptic curve $[E,y,z] \in \mathcal{M}_{1,2}$ satisfying $(2k-1)y \neq (2k-1)z$ as an equivalence of divisors and $[C\cup_{x\sim y}E,z] \in \overline{Z}_2$ for a generic element $[C,x]$ of $Z_1$. Then a generic element $[X,p]$ of $Z_2$ does not admit a $g^1_{i+k}$ with ramification order at $p$ greater or equal to $2k-1$.  
\end{prop} 
\begin{proof}
	We assume by contradiction that a generic element $[X,p]$ in $Z_2$ admits such a $g^1_{i+k}$. Then over all points $[C\cup_{x\sim y}E,z]$ as in the hypothesis there exists an admissible cover $f\colon Y \rightarrow \Gamma$ of degree $i+k$ and having ramification order greater or equal to $2k-1$ at a point collapsing to $z$ when we stabilize $Y$. 
	
	Denote by $l$ the natural number such that $\mathrm{deg}(f_{|C}) = i+k-l$. The existence of the ramification point implies the existence of an irreducible component $\mathbb{P}$ of $\Gamma$ such that the degree of $f_{|Y_{x;C}}$ over it is at least $2k-1$. 
	
	Applying Lemma \ref{admcyc} we get that $\mathrm{ord}_x(f_{|C}) \geq 2k-1-l$. If $l \neq 0$ we cannot have such a map $f_{|C}$ for a generic $[C,x] \in Z_1$ because of the hypothesis. 
	
	The only possibility left is $l=0$. Assume that $\mathrm{ord}_x(f_{|C}) = 2k-1$. As $\mathrm{deg}(f_{|C}) = i+k$, Lemma \ref{admcyc} implies that the image of $f_{|Y_{x;C}}$ is the union of the components $\mathbb{P}$ of $\Gamma$ in the same direction as $x$ with respect to $C$. 
	
	Applying Lemma \ref{admcyc} we get that the rational curves in $Y_{z;E}$ have a contribution of $2k-1-\mathrm{ord}_z(f_{|E})$ for the degree of $f$ over the target of $E$. It follows that the rational curves in $Y_{z;E}$ have a contribution of $2k-1-\mathrm{ord}_z(f_{|E})$ for the degree of $f$ over the target of $C$, which implies $\mathrm{ord}_z(f_{|E}) = 2k-1$.
	
	Let $Z$ be the other irreducible component of $Y$ containing the node $y$. The same reasoning as above for the curve $Y_{y;Z}$ implies $\mathrm{ord}_y(f_{|E}) = 2k-1$. Clearly $\mathrm{deg}(f_{|C}) = 2k-1$ and hence we have the equivalence of divisors $(2k-1)y = (2k-1)z$ which we know is not true. 
		
	It follows that $\mathrm{deg}(f_{|C}) = i+k$ and $\mathrm{ord}_x(f_{|C}) \geq 2k$, contradicting the hypothesis. Hence, our assumption was wrong and the proposition follows.  
	
\end{proof}

Next, we provide a counterpart to Proposition \ref{eventoodd} when the curves in $Z_1$ and $Z_2$ have odd and even genus, respectively. We should remark that although there are slight numerical differences between Proposition \ref{eventoodd} and Proposition \ref{oddtoeven}, the method of the proof is the same. Thus, by degenerating to the boundary of $Z_2$ in order to get a contradiction with the properties of the curves in $Z_1$, we obtain:
\begin{prop} \label{oddtoeven} Let $k\geq 1, i \geq 2$ and consider two irreducible subspaces $Z_1 \subseteq \mathcal{M}_{2i-1,1}$ and $Z_2 \subseteq \mathcal{M}_{2i,1}$ such that a generic element $[C,x]$ of $Z_1$ does not admit a $g^1_{i+j}$ with ramification order at $x$ greater or equal to $2j+1$ for $0\leq j \leq k$. Moreover, assume there exists a 2-pointed elliptic curve $[E,y,z] \in \mathcal{M}_{1,2}$ satisfying $2ky \neq 2kz$ as an equivalence of divisors and $[C\cup_{x\sim y}E,z] \in \overline{Z}_2$ for a generic element $[C,x]$ of $Z_1$. Then a generic element $[X,p]$ in $Z_2$ has maximal gonality and does not admit a $g^1_{i+k}$ with ramification order at $p$ greater or equal to $2k$.  
\end{prop}

\section{Limit linear series on strata of differentials}
Let $g \geq 3$ and $\mu$ a positive, length $n$ partition of $2g-2$. Our first task is to reduce Theorem \ref{trmabel} to the study of only one case, namely $\mu = (2g-2)$. This is done via an argument by specialization. 

Let $[\mathbb{P}^1, y, y_1, \ldots, y_n]$ be an $n+1$-pointed rational curve and consider the morphism 
\[ \mathcal{H}_g(2g-2)\rightarrow \overline{\mathcal{M}}_{g,n} \]
sending a pointed curve $[C,x]\in\mathcal{H}_g(2g-2)$ to $[C\cup_{x\sim y}\mathbb{P}^1, y_1\ldots y_n]$.
As a consequence of \cite[Corollary 1.4]{DaweiAbComp}, the space $\mathcal{H}_g(2g-2)$ is mapped by this morphism into the compactification of $\mathcal{H}_g(\mu)$ inside $\overline{\mathcal{M}}_{g,n}$. Moreover, in the case $\mu =(2m_1,\ldots, 2m_n)$, when $\mathcal{H}_g(\mu)$ has two nonhyperelliptic components $\mathcal{H}^{\mathrm{odd}}_g(\mu)$ and $\mathcal{H}^{\mathrm{even}}_g(\mu)$ depending on the parity of the spin structure, we see that the morphism preserves the parity of the component. As a consequence, every nonhyperelliptic component of a stratum $\mathcal{H}_g(\mu)$ contains at least one of $\mathcal{H}^{\mathrm{even}}_g(2g-2)$ or $\mathcal{H}^{\mathrm{odd}}_g(2g-2)$ in its compactification. 

When $k=2$ and the length $n$ partition $\mu$ of $4g-4$ is positive, we denote $\mathcal{Q}^2_g(\mu) \coloneqq \mathcal{H}^2_g(\mu)\setminus\mathcal{H}_g(\frac{\mu}{2})$ (here $\mathcal{Q}^2_g(\mu) \coloneqq \mathcal{H}^2_g(\mu)$ if $\mu$ has at least one odd entry) and we consider the morphism
\[ \mathcal{Q}_g(4g-4) \rightarrow \overline{\mathcal{M}}_{g,n}  \]
mapping $[C,x] \in \mathcal{Q}_g(4g-4)$ to $[C\cup_{x\sim y}\mathbb{P}^1, y_1\ldots y_n]$. Because of \cite[Theorem 1.5]{Daweik-diffcomp}, the image of this map is contained in the closure of $\mathcal{Q}_g(\mu)$ inside $\overline{\mathcal{M}}_{g,n}$. Lastly, when $g=4$ and $\mathcal{Q}_4(\mu)$ has two nonhyperelliptic components $\mathcal{Q}^{\mathrm{irr}}_4(\mu)$ and $\mathcal{Q}^{\mathrm{reg}}_4(\mu)$ , Proposition 7.4 in \cite{ChenM} implies that the morphism maps  $\mathcal{Q}^{\mathrm{irr}}_4(12)$ and  $\mathcal{Q}^{\mathrm{reg}}_4(12)$ into $\overline{\mathcal{Q}}^{\mathrm{irr}}_4(\mu)$ and $\overline{\mathcal{Q}}^{\mathrm{reg}}_4(\mu)$, respectively.

As having maximal gonality is an open condition, it is enough to prove Theorems \ref{trmabel} and \ref{trmquad} for the length one partition. 

We will treat first the case of canonical strata. To simplify notation we denote by $\mathcal{H}^{\mathrm{nonhyp}}_g(2g-2)$ the union of the nonhyperelliptic components of $\mathcal{H}_g(2g-2)$. Moreover, by a generic point of a space with more than one component we mean a generic point of any of the components.

\begin{prop} \label{subcanonical} In the even case $g=2i$, a generic element of $\mathcal{H}_{2i}^{\mathrm{nonhyp}}(4i-2)$ has maximal gonality and, for any $1\leq j \leq i-2$, does not admit a $g^1_{i+j}$ with ramification order at the marked point greater than or equal to $2j$. Moreover, a generic element of $\mathcal{H}_{2i}^{\mathrm{odd}}(4i-2)$ does not admit a $g^1_{2i-1}$ with ramification order at the marked point greater than or equal to $2i-2$. 
	
	In the odd case $g=2i+1$, for any $1\leq j \leq i-1$, a generic element of $\mathcal{H}_{2i+1}^{\mathrm{nonhyp}}(4i)$ does not admit a $g^1_{i+j}$ with ramification order at the marked point greater than or equal to $2j-1$. Moreover, a generic element of $\mathcal{H}_{2i+1}^{\mathrm{odd}}(4i)$ does not admit a $g^1_{2i}$ with ramification order at the marked point greater than or equal to $2i-1$.   
\end{prop}
\begin{proof} 
	We prove the proposition by induction on the genus. The initial case of the stratum $\mathcal{H}_3^{\mathrm{nonhyp}}(4)$ is clear. We assume the proposition to be true up to genus $g-1$ and prove it for genus $g$. First, consider the case when the genus $g=2i$ is even.  
	
	We consider the clutching morphism
	\[ \pi\colon \mathcal{H}_{2i-1}^{\mathrm{odd}}(4i-4)\times \mathcal{H}_1(4i-2,-4i+2)\rightarrow \overline{\mathcal{H}}_{2i}^{\mathrm{nonhyp}}(4i-2) \subseteq \overline{\mathcal{M}}_{2i,1}\]
	obtained by glueing the marking of the first component to the second marked point in the elliptic stratum. Using the compactification of the strata of differentials appearing in \cite{DaweiAbComp}, we see that the image of $\pi$ lies in the specified target. 
	
	 Because for any $1 \leq j \leq i-2$ the stratum $\mathcal{H}_1(4i-2,-4i+2)\setminus \mathcal{H}_1(2j,-2j)$ has both odd and even components, it follows that the image of $\pi$ intersect both components of $\mathcal{H}_{2i}^{\mathrm{nonhyp}}(4i-2)$. The induction hypothesis assures that all conditions in Proposition \ref{oddtoeven} are satisfied. We conclude that a generic element of $\mathcal{H}_{2i}^{\mathrm{nonhyp}}(4i-2)$ has maximal gonality and does not admit a $g^1_{i+j}$ with ramification order greater or equal to $2j$ at the marking. 
	 
	 We are left to prove that a generic point $[C,x]$ of $\mathcal{H}_{2i}^{\mathrm{odd}}(4i-2)$ does not admit a $g^1_{2i-1}$ with ramification order at $x$ greater than or equal to $2i-2$. Assume that such a $g^1_{2i-1}$ exists. Hence there exists an $y\in C$ such that $h^0(C, (2i-2)x+y) \geq 2$. Moreover, as $[C,x]$ does not admit a $g^1_{2i-2}$ with ramification order greater or equal to $2i-4$ at $x$, it follows that $h^0(C,(2i-2)x) = 1$. Applying the Riemann-Roch theorem we deduce that 
	 \[h^0(C,2ix) = 2 \textrm{ \ and \ } h^0(C,2ix-y)\geq 2 \]
	 It follows that $y$ is a base-point of the linear system $|2ix|$, hence $y = x$ and $h^0(C, (2i-1)x) = 2$ which is clearly false in the odd-spin component. It follows from the contradiction that our assumption was wrong and $[C,x]$ does not admit a $g^1_{2i-1}$ with ramification order at $x$ greater than or equal to $2i-2$. This completes the case $g=2i$. 
	 
	 For the case $g=2i+1$ we consider the clutching 
	\[ \pi\colon \mathcal{H}_{2i}^{\mathrm{odd}}(4i-2)\times \mathcal{H}_1(4i,-4i)\rightarrow \overline{\mathcal{H}}_{2i+1}^{\mathrm{nonhyp}}(4i-2) \subseteq \overline{\mathcal{M}}_{2i,1} \]
	obtained by glueing the marking of the first component to the second marked point in the elliptic stratum. Again, the compactification of the strata of differentials appearing in \cite{DaweiAbComp} ensures that the image is in the specified target. Everything follows analogously to the case $g=2i$.
    
\end{proof}

A similar approach for the case of $2$-canonical divisors implies the following: 
\begin{prop}  \label{quadsubcan}
		A generic point $[C,x]$ of $\mathcal{Q}_{2i+1}(8i) \coloneqq \mathcal{H}^2_{2i+1}(8i) \setminus \mathcal{H}_{2i+1}(4i)$ does not admit a $g^1_{i+j}$ with ramification order greater or equal to $2j-1$ at $x$, for any $1\leq j \leq i-1$. For even genus, a generic element $[C,x]$ of $\mathcal{Q}_{2i}(8i-4)$ has maximal gonality and does not admit a $g^1_{i+j}$ with ramification order greater or equal to $2j$ at $x$, for any $1\leq j \leq i-3$.
\end{prop}
\begin{proof} 
	We define the subspace $\mathcal{R}_1^2(8i,-8i) \subseteq \mathcal{H}^2_1(8i,-8i)\setminus \mathcal{H}_1(4i,-4i)$ to be the space parametrizing points $[E,x,y]$ such that if $s$ is a quadratic differential satisfying $\mathrm{div}(s) = 8ix-8iy$, then $\mathrm{Res}^2_y(s) = 0$, where $\mathrm{Res}^2_y(s)$ is defined to be the second power of $\mathrm{Res}_y(\omega)$ where $\omega$ is a meromorphic abelian differential at $y$ satisfying locally $s = \omega^2$. We know from \cite[Theorem 1.2]{GenTah} that the space  $\mathcal{R}_1^2(8i,-8i)$ is non-empty. 
	
	The proof of Proposition \ref{subcanonical} can  be adapted to the case of quadratic strata by considering the clutching morphism
	\[ \pi\colon \mathcal{H}^{\mathrm{nonhyp}}_{2i}(4i-2)\times \mathcal{R}^2_1(8i,-8i) \rightarrow \overline{\mathcal{Q}}_{2i+1}(8i) \subseteq \overline{\mathcal{M}}_{2i+1,1} \]
	where the marking of the first component is glued to the second marked point of the stratum $\mathcal{R}^2_1(8i,-8i)$. The fact that the image is in the specified target follows from the description of the compactification in \cite{Daweik-diffcomp}. Moreover, the fact that $\mathcal{R}_1^2(8i,-8i) \subseteq \mathcal{H}^2_1(8i,-8i)\setminus \mathcal{H}_1(4i,-4i)$ ensures that $\mathcal{R}_1^2(8i,-8i)$ does not intersect $\mathcal{H}_1(2j-1,-2j+1)$ for every $1\leq j\leq i-1$.
	
	The space $\mathcal{Q}_{2i+1}(8i)$ is irreducible for all $i\geq 1$ hence the image of the map $\pi$ intersects all the (one) components. By applying Proposition \ref{subcanonical} for $\mathcal{H}^{\mathrm{nonhyp}}_{2i}(4i-2)$ and Proposition \ref{eventoodd} for the spaces $Z_1 = \mathcal{H}^{\mathrm{nonhyp}}_{2i}(4i-2)$ and $Z_2 = \mathcal{Q}_{2i+1}(8i)$ we get the desired conclusion for odd genus. 
	
	For even genus, we consider the clutching 
	\[\pi\colon \mathcal{Q}_{2i-1}(8i-8)\times \mathcal{H}_1(8i-4,-8i+4) \rightarrow \overline{\mathcal{Q}}_{2i}(8i-4) \subseteq \overline{\mathcal{M}}_{2i,1}\]  
	where we glue the marking of the first component to the second marked point of the stratum $\mathcal{H}_1(8i-4,-8i+4)$. Again the description of the compactification in \cite{Daweik-diffcomp} implies this is well defined. The conclusion follows from Proposition \ref{oddtoeven} for $i\geq 3$ when we know the space $\mathcal{Q}_{2i}(8i-4)$ to be irreducible. When $i=2$ the stratum $\mathcal{Q}_4(12)$ has two irreducible components. The fact that a generic point of either component of $\mathcal{Q}_{4}(12)$ does not admit a $g^1_2$, see \cite[Lemma 7.7]{ChenM}, completes the proof.
\end{proof}	
Observe that the specialization argument and Proposition \ref{quadsubcan} imply Theorem \ref{trmquad} for all genera $g$ except $g=3$, case for which the proposition is empty. As such, we consider this case separately. 
\begin{prop} \label{genus3}
	For a generic point $[C,x]$ of $\mathcal{Q}_3(8)$ the underlying curve $C$ is trigonal. 
\end{prop}
\begin{proof}
	To conclude the proposition, it is enough to exhibit a unique element $[C,x]$  of $\mathcal{Q}_3(8)$ with underlying trigonal curve. For this, we consider the quartic $C = V(x^4-x^2z^2+yz^3-2y^3z)$ and the quadric $Q = V(x^2 -y^2-yz)$ in $\mathbb{P}^2$. We show that $C$ is smooth and they intersect only in the point $[0:0:1]$. Moreover, any line $L$ passing through $[0:0:1]$ intersects $C$ in another point. 
	
	We start by showing that $C$ is smooth. To see that $C$ is smooth at a point $[x:y:z]$ we use the projective Jacobi criterion for smoothness that says $[x:y:z] \in \mathrm{Sing}(C)$ if and only if we have the vanishing of the partial differentials at the point: 
	\[ x^4 -x^2z^2 + yz^3- 2y^3z = 0, \ \ 4x^3-2xz^2=0, \ \ z^3-6y^2z = 0 \ \mathrm{and} \ 3z^2y-2y^3-2zx^2 = 0 \]
	We assume that all four equations are simultaneously satisfied. The second one implies that either $x=0$ or $2x^2 = z^2$. 
	
	If $x= 0$, the other three equations can be rewritten as 
	\[ yz(z^2-2y^2)= 0, \ \ z(z^2-6y^2) = 0 \ \mathrm{and} \ y(3z^2-2y^2) = 0\]
	This system has the unique solution $y=z =0$. It follows there is no singular point of the form $[0:y:z]$ in C.   
	
	If $2x^2 = z^2 \neq 0$, the third equation implies $z^2 = 6y^2$. Substituting $x^2=3y^2$ and $z^2=6y^2$ in the first equation we get 
	\[ 9y^4- 18y^4 +6y^3z-2y^3z = 0\]
	This implies $4z = 9y$ which is obviously not possible. 
	
	We now prove that the point $[0:0:1]$ is the unique point in $C\cap Q$. By substituting $x^2=y^2 +yz$ in the equation of the quartic we get $y^4= 0$. As $x^2 = y^2 +yz$ we get that $x=0$, proving the uniqueness of the point. B\'ezout's theorem implies that the point $[0:0:1]$ has multiplicity 8 in the intersection. Moreover, it is obvious that the intersection $C \cap V(ax-by)$ contains a point different from $[0:0:1]$ for any choice $[a:b] \in \mathbb{P}^1$. 
	
	The adjunction formula implies $\omega_C \cong \mathcal{O}_C(1)$. As a consequence, $[C,[0:0:1]]$ is an element of $\mathcal{Q}_3(8)$ and $C$ is trigonal as the canonical map defines an embedding.  
	\end{proof}

We deduce from Propositions \ref{subcanonical}, \ref{quadsubcan} and \ref{genus3} that Theorem \ref{trmabel} and Theorem \ref{trmquad} are true for the partition of length one. By specialization, we conclude this is true for every nonhyperelliptic component of an abelian or quadratic stratum in the respective genus, hence proving Theorem \ref{trmabel} and Theorem \ref{trmquad}. 

Theorem \ref{unirul} follows as a trivial consequence of Theorem \ref{trmabel} and \cite[Remark 1.4]{BAR18}. 

\begin{rmk} \normalfont 
	The study of the generic element $[C,x]$ in the stratum $\mathcal{H}_g(2g-2)$ of curves with a subcanonical point appears in the literature in \cite{subcanbul}. The main result there describes the Weierstrass gap of the point $x$ for all three components of $\mathcal{H}_g(2g-2)$. Proposition \ref{subcanonical} provides another proof of \cite[Theorem 2.1]{subcanbul}.
\end{rmk} 

Let $Z$ be an irreducible component of $\mathcal{H}^k_g(\mu)$ and consider the forgetful map $\pi_Z\colon Z \rightarrow \mathcal{M}_g$. Corollary 5 in \cite{Bud} implies that a generic fibre of the map $\pi_Z$ has dimension $h^0(C,p_1+\ldots+p_n)-1$ where $[C,p_1,\ldots,p_n]$ is a generic point of $Z$. 

\begin{rmk}
	Theorem \ref{trmabel} and Theorem \ref{trmquad} provide another proof of the fact that $\pi_Z$ has generically finite fibers when $k =1$ or $k=2$, the component $Z$ is nonhyperelliptic and the partition is positive of length $l(\mu) \leq \frac{g+1}{2}$.
\end{rmk}

\section{Strata of $2$-canonical divisors in low genus}

Consider the blow-up of $\mathbb{P}^2$ in $0\leq r\leq 8$ points. The resulting surface $S$ is a del Pezzo surface and the class $|-2K_S|$ is ample. There exists a moduli space $\mathcal{P}_r$ parametrizing such surfaces. Moreover, we can consider the space: 
\[ \mathcal{B}_r = \left\{ (S,C) \ | \ S\in\mathcal{P}_r \ \mathrm{and} \ C \in |-2K_S| \textrm{ \ a smooth and irreducible curve } \right\}\]

The adjunction formula implies the following equality for the genus $g$ of such a curve $C$: 
\[ 2g-2 = C\cdot C + K_S\cdot C  = 18-2r\]
Following the method in \cite{BAR18} we see that to prove uniruledness of an irreducible component $Z$ of $\mathcal{H}^2_g(\mu)$ in the range $3\leq g \leq 6$ it is enough to prove that the underlying curve of a generic element of $Z$ is in the image of the forgetful morphism 
\[ \psi_r\colon\mathcal{B}_r \rightarrow \mathcal{M}_{10-r}\] 

We know from \cite[Proposition 2.1]{BAR18} that $\psi_r$ is dominant when $4\leq r\leq 7$ and that the curves having a plane nodal model of degree 6 are in the image of $\psi_r$.

\begin{prop}
	Let $3\leq g\leq 5$ and $C$ a genus $g$ curve of maximal gonality. Then $C$ is in the image of the map $\psi_{10-g}$. When $g=6$ the image of $\psi_4$ contains all curves $C$ of maximal gonality that are not bi-elliptic and not plane quintics.
\end{prop} 
\begin{proof}
	If a curve can be embedded as a degree 6 curve in $\mathbb{P}^3$, we obtain a plane nodal model of degree 6 by projecting from a generic point of $\mathbb{P}^3$, see \cite[Theorem 3.10]{Hartshorne}. Hence it is enough to provide on $C$ a very ample line bundle $A$ with $h^0(C,A) = 4$ and $\mathrm{deg}(A) = 6$. Consider $x$ and $y$ two generic points of $C$. When $g=3$ we consider $A = \omega_C(x+y)$ and when $g=4$ we simply take $A = \omega_C$. In these cases the Riemann-Roch theorem implies our conclusion.  
	
	If $g=5$ and $C$ is a curve of maximal gonality, it is a classical result in algebraic geometry that its canonical model is the complete intersection of three quadrics $Q_1, Q_2$ and $Q_3$ in $\mathbb{P}^4$. Using Lemma 2.5 in \cite{Ide}, it follows we can choose the quadrics $Q_1$ and $Q_2$ in such a way that $Q_1\cap Q_2$ is a smooth surface which we denote $S$. The adjunction formula for complete intersections imply $\omega_S\cong \mathcal{O}_S(-1)$. In particular, $S$ is a del Pezzo surface as the anticanonical bundle is ample. Moreover $C \in |-2K_S|$ and hence the conclusion follows. 
	
	In genus $g=6$, we know from \cite{Mukgrass} that a curve $C$ of maximal gonality that is not bi-elliptic and not a plane quintic is the intersection in $\mathbb{P}^9$ of the Grassmannian $G(2,5)$ with four hyperplanes $H_1, H_2,H_3, H_4$ and a hyperquadric $Q$. Here the Grassmannian is embedded in $\mathbb{P}^9$ by the Pl\"ucker coordinates. 
	
	In particular, in the variety $X = G(2,5)$, the curve $C$ is the complete intersection of 5 divisors, denoted again $H_1,H_2, H_3, H_4$ and $Q$ accordingly. If we consider the class of the divisor $H_i - Q$ we see it is a negative multiple of an ample class. This follows from the fact that $\mathcal{O}_{\mathbb{P}^9}(Q-H_i)\cong \mathcal{O}_{\mathbb{P}^9}(1)$ and the restriction of an ample class to a subvariety is ample. As a consequence, the linear system $|H_i-Q|$ is empty and the conditions of Lemma 2.5 in \cite{Ide} are trivially satisfied. Hence we can assume that $G(2,5) \cap H_1 \cap\ldots \cap H_4$ is a smooth surface $S$. Using the adjunction formula for the complete intersection $S$ inside $X$ we obtain 
	\[\omega_S \cong \omega_X\otimes \mathcal{O}_S(4)\]
	We know from Proposition 1.9 in \cite{Mukgrass} that $\omega_X \cong \mathcal{O}_X(-5)$ and as a consequence $\omega_S \cong \mathcal{O}_S(-1)$. It follows that $S$ is a del Pezzo surface and $C \in |-2K_S|$, hence the conclusion follows. 

\end{proof} 

\textbf{Proof of Theorem 1.4:} The statement is clear for genus $3\leq g\leq5$. When $g=6$ and $l(\mu) \geq 4$ we want to show that a generic curve in the image of $\pi_\mu\colon \mathcal{Q}_6(\mu) \rightarrow\mathcal{M}_6$ has maximal gonality, is not bi-elliptic and is not a plane quintic. The locus of smooth plane quintics is 12 dimensional while the locus of bi-elliptic curves is 10-dimensional. We know from \cite{Bud} that the image of $\pi_\mu$ has dimension $\mathrm{min}\left\{9+l(\mu), 15\right\}$, which completes the proof.  

\hfill $\square$
\bibliography{main}
\bibliographystyle{alpha}

\end{document}